\documentclass[12pt, reqno]{amsart}
\usepackage{amsmath, amsthm, amscd, amsfonts, amssymb, graphicx, color}
\usepackage{enumerate}
\usepackage[bookmarksnumbered, colorlinks, plainpages]{hyperref}
\hypersetup{colorlinks=true,linkcolor=red, anchorcolor=green, citecolor=cyan, urlcolor=red, filecolor=magenta, pdftoolbar=true}

\newtheorem{theorem}{Theorem}[section]
\newtheorem{lemma}[theorem]{Lemma}
\newtheorem{proposition}[theorem]{Proposition}

\theoremstyle{definition}
\newtheorem{definition}[theorem]{Definition}
\newtheorem{example}[theorem]{Example}

\theoremstyle{remark}
\newtheorem{remark}[theorem]{Remark}
\numberwithin{equation}{section}

\newcommand{\tref}[1]{Theorem~\textup{\ref{#1}}}

\newcommand{\cref}[1]{Corollary~\textup{\ref{#1}}}
\newcommand{\lref}[1]{Lemma~\textup{\ref{#1}}}
\newcommand{\rref}[1]{Remark~\textup{\ref{#1}}}

\newenvironment{psmallmatrix}{\left(\begin{smallmatrix}}{\end{smallmatrix}\right)}
\begin{document}
\setcounter{page}{1}

\title[Symmetric N.Ineq.]{On Symmetric Norm Inequalities And Hermitian Block-Matrices }
\author[A.Mhanna]{Antoine Mhanna$^1$$^{*}$}

\address{$^{1}$ Department of Mathematics, Lebanese University  Hadath, Beirut, Lebanon}
\email{\textcolor[rgb]{0.00,0.00,0.84}{tmhanat@yahoo.com}}



\date{21 November 2015
\newline \indent $^{*}$ Corresponding author}

\begin{abstract}The main purpose of this paper  is to englobe some  new and known types of Hermitian block-matrices $M=\begin{pmatrix} A & X\\ {X^*} & B\end{pmatrix}$  satisfying or not the inequality $\|M\|\le \|A+B\|$ for all symmetric norms.
\end{abstract}

\maketitle

\section{Introduction and preliminaries}
The first section will deal with some known results on the inequality and some preliminaries we  used  in the second section to derive some new  generalization results. For positive semi-definite block-matrix  $M,$ we say that $M$ is P.S.D. and  we write  $M=\begin{pmatrix} A & X\\ {X^*} & B\end{pmatrix} \in {\mathbb{M}}_{n+m}^+$, with $A\in {\mathbb{M}}_n^+$, $B \in {\mathbb{M}}_m^+.$ 
Let $A$ be an $n\times n$ matrix  and  $F$ an  $m\times m$ matrix, $(m>n)$ written by blocks such that  $A$ is a diagonal block and all entries other than those of $A$ are zeros, then the two matrices have the same  singular values and  for all unitarily invariant norms $\|A\|=\|F\|=\|A\oplus0\|$, we say then that the  symmetric norm on ${\mathbb{M}}_{m}$ induces a symmetric norm on ${\mathbb{M}}_{n}$, so for square matrices  we may assume that our norms are defined on all spaces ${\mathbb{M}}_n,$           $  n\ge 1.$
The spectral norm is denoted by ${\|.\|}_s,$ the Frobenius norm by ${\|.\|}_{(2)},$ and the Ky Fan $k-$norms by ${\|.\|}_{k}.$
Let  $Im(X):=\dfrac{X-X^*}{2i}$ respectively $Re(X)=\dfrac{X+X^*}{2}$ be  the imaginary part respectively the real part of a matrix $X$ and let ${\mathbb{M}}_n^+$ denote the set of positive and semi-definite part of the space of $n\times n$ complex matrices and $M$  be any positive semi-definite block-matrices; that is, 
$M=\begin{bmatrix} A & X\\ {X^*} & B\end{bmatrix} \in {\mathbb{M}}_{n+m}^+$, with $A\in {\mathbb{M}}_n^+$, $B \in {\mathbb{M}}_m^+.$
\begin{lemma}\label{la}\cite{1}
For every matrix in ${\mathbb{M}}_{2n}^+ $ written in blocks of the same size, we have the  decomposition: 
$$\begin{pmatrix} A & X\\ {X^*} & B\end{pmatrix}=U\begin{pmatrix} \frac{A+B}{2}+Im(X) & 0\\ {0} & 0 \end{pmatrix}U^*+V\begin{pmatrix} 0 & 0\\ {0} &\frac{A+B}{2}-Im(X) \end{pmatrix}V^*$$
for some unitaries $U, V\in {\mathbb{M}}_{2n}.$
\end{lemma}
\begin{lemma}\label{lk}\cite{1}
For every matrix in ${\mathbb{M}}_{2n}^+ $ written in blocks of the same size, we have the  decomposition: 
$$\begin{pmatrix} A & X\\ {X^*} & B\end{pmatrix}=U\begin{pmatrix} \frac{A+B}{2}+Re(X) & 0\\ {0} & 0 \end{pmatrix}U^*+V\begin{pmatrix} 0 & 0\\ {0} &\frac{A+B}{2}-Re(X) \end{pmatrix}V^*$$
for some unitaries $U, V\in {\mathbb{M}}_{2n}.$
\end{lemma}
\begin{remark}\label{pl}
The proofs of \lref{la} respectively \lref{lk} suggests that we have ${A+B}\ge -\dfrac{(X-X^*)}{i}$ and ${A+B}\ge \dfrac{(X-X^*)}{i},$ respectively ${A+B}\ge -(X+X^*)$ and ${A+B}\ge (X+X^*).$
\end{remark}
\begin{lemma}\cite{4}\label{sinj}
Let $M=\begin{bmatrix}A&B\\C&D\end{bmatrix}$ be any square matrix written by blocks of same size, if $AC=CA$ then $\det(M)=\det(AD-CB).$
\end{lemma}
\section{Main results}
\subsection{Symmetric norm inequality}
Hereafter our block matrices are such their   diagonal blocks are of equal size.
\begin{lemma}\label{r} \cite{1} Let $M=\begin{bmatrix} A & X\\ {X^*} & B\end{bmatrix} \in {\mathbb{M}}_{2n}^+$,  if $X$ is Hermitian or Skew-Hermitian then  \begin{equation}\label{pkm}\|M\|\le \|A+B\|\end{equation} for all symmetric norms. \end{lemma}
The fact is that there exist P.S.D. matrices with non Hermitian or Skew-Hermitian off-diagonal blocks satisfying \eqref{pkm}.
\begin{definition}A block matrix $N=\begin{bmatrix} A & X\\ {X^*} & B\end{bmatrix} $ is said to be a Hermitio matrix  if it is unitarly congruent to a matrix $M=\begin{bmatrix} A & Y\\ {Y} & B\end{bmatrix}$ with Hermitian  off diagonal blocks or to $M=\begin{bmatrix} A & Y\\ {-Y} & B\end{bmatrix}$ with  Skew-hermitian off diagonal blocks.
\end{definition}
Clearly if $N$ is P.S.D. it satisfies $\|N\|\le \|A+B\|$ for all symmetric norms (by \lref{r}).
\begin{proposition}
Let    $M=\begin{bmatrix}A&X\\X^*&B    \end{bmatrix}\in {\mathbb{M}}_{2n}^+$ be  a given positive semi-definite matrix. If  $X^*$ commute with $A$, or   $X$ commute with $B$, then  $\text{           }  \|M\| \le \|A+B\| \text{         }$ for all symmetric norms.
\end{proposition}
\begin{proof}
We will assume  without loss of generality that $X^*$ commute with $A,$ as the other case is similar. We show that such $M$ is a Hermitio matrix. Take the right  polar decomposition of $X^*$  so  $X=U|X|$ and $X^*=|X|U^*.$ Since   $U^*$ is unitary  and $X^*$ commute with $A,$  $X$ and $|X|$ commute with $A$ thus $AU^*=U^*A.$ If $I_n$ is the identity matrix of order $n$, a direct computation shows that $$\begin{bmatrix} U^*& 0\\ {0} & I_n\end{bmatrix}  \begin{bmatrix} A & X\\ {X^*} & B\end{bmatrix}\begin{bmatrix} U & 0\\ {0} & I_n\end{bmatrix}  =\begin{bmatrix} A & |X|\\ {|X|} & B\end{bmatrix}   ,$$
consequently by \lref{r}, $\|M\|\le \|A+B\|$ for all symmetric norms and that completes the proof.
\end{proof}

\begin{remark}It is easily seen that if $X$ commute with the Hermitian matrix $A$ so is $X^*$ and conversely.
\end{remark}
Let $M_1=\dfrac{A+B}{2}+Im(X),$    $M_2=\dfrac{A+B}{2}-Im(X),$   $N_1=\dfrac{A+B}{2}+Re(X)$  and   $N_2=\dfrac{A+B}{2}-Re(X).$   
The following is  a slight generalization  of \lref{r} unless one proves that this is a case of a Hermitio matrix.
\begin{theorem}\label{K}Let $M= \begin{bmatrix} A & X\\ {X^*} & B\end{bmatrix}$ be  a positive semi definite matrix,  if  $Im(X)=rI_n$ or $Re(X)=rI_n$  for some $r,$ then $\|M\|\le \|A+B\|$ for all symmetric norms.
\end{theorem}\begin{proof}Let ${\sigma}_i (H)$ denote the singular values of a matrix $H$ ordered in decreasing order,  by \rref{pl}   the matrices    $M_1=\dfrac{A+B}{2}+Im(X)$ and    $M_2=\dfrac{A+B}{2}-Im(X)$     are positive semi definite since $Im(X)=rI_n$ we have:  $$\sum_{i=1}^k{\sigma}_i\left(\dfrac{A+B}{2}+Im(X)\right)+\sum_{i=1}^k{\sigma }_i\left(\dfrac{A+B}{2}-Im(X)\right)= \sum_{i=1}^k{\sigma }_i(A+B).      $$
In other words by \lref{la} $\|M\|_k \le \|M_1\|_k+\|M_2\|_k= \|A+B\|_k$ for all Ky-Fan $k-$norms which completes the proof, using \lref{lk} the other case is similarly proven.
\end{proof}\begin{theorem}Let $M= \begin{bmatrix} A & X\\ {X^*} & B\end{bmatrix}$ be  a positive semi definite matrix,  then $\|M\|\le 2\|A+B\|$ for all symmetric norms. Furthermore if $M_1$  or $M_2$ are positive definite  then the large inequality is a strict one. 
\end{theorem}
\begin{proof}The proof is very close to that of the previous theorem since $M_1,$    $M_2$   are two  positive semi definite matrices we  have $\|M_1\|_k\le \|A+B\|_k$ and $\|M_2\|_k \le \|A+B\|_k$ for all $k\le n,$ thus we derive    the following inequality: $$   \|M\|_k \le \|M_1\|_k+\|M_2\|_k \le  2\|A+B\|_k  $$ for all Ky-Fan $k-$norms. It is easily seen that if $M_1$ or $M_2$ are P.D. then $ \|M_1\|_k+\|M_2\|_k <  2\|A+B\|_k  $
\end{proof}
\begin{remark}Note  by the decompositions in  \lref{la} and \lref{lk} if $M>0$  then  all of $M_1,$ $N_1$,   $M_2$ and $N_2$ are positive definite.
\end{remark}

Here's a counter example: 
\begin{example}
Let $$T=\begin{bmatrix}\dfrac{3}{10}&0&0&\dfrac{i}{11}\\0&5&i&0\\0&-i&5&0\\\dfrac{-i}{11}&0&0&\dfrac{3}{10}\end{bmatrix}=\begin{bmatrix}A&X\\X^*&B\end{bmatrix}$$
where $X=\begin{bmatrix}0&\frac{i}{11}\\i& 0   \end{bmatrix}$, $A=\begin{bmatrix}\frac{3}{10}&0\\0&5    \end{bmatrix}$ and $B=\begin{bmatrix}5&0\\0&\frac{3}{10} \end{bmatrix}.$                      The eigenvalues of $T$ are the positive numbers: $\displaystyle{\lambda}_1=6,\text{   } {\lambda}_2=4,$ ${\lambda}_3\approx~0.39,$ $ \text{ }{\lambda}_4\approx 0.21,$ $T\ge 0,$  but $
6={\|T\|}_{s} >{\|A+B\|}_s=\frac{53}{10} .$\end{example} 
\begin{lemma}\label{plp}
Let $$N=\begin{bmatrix}\begin{psmallmatrix}a_{1} & 0 &\cdots & 0\\ 0 & a_{2} & \cdots & 0\\ \vdots & \vdots & \ddots &\vdots \\
0& 0 & \cdots & a_{n} \end{psmallmatrix}&D\\D^*&\begin{psmallmatrix}b_{1} & 0 &\cdots & 0\\ 0 & b_{2} & \cdots & 0\\ \vdots & \vdots & \ddots &\vdots \\
0& 0 & \cdots & b_{n} \end{psmallmatrix}\end{bmatrix}$$ where $a_1,\cdots, a_{n}$ respectively $b_{1},\cdots, b_{n}$ are  nonnegative respectively  negative real numbers, $A=\begin{psmallmatrix}a_{1} & 0 &\cdots & 0\\ 0 & a_{2} & \cdots & 0\\ \vdots & \vdots & \ddots &\vdots \\0& 0 & \cdots & a_{n} \end{psmallmatrix},$ $B=\begin{psmallmatrix}b_{1} & 0 &\cdots & 0\\ 0 & b_{2} & \cdots & 0\\ \vdots & \vdots & \ddots &\vdots \\
0& 0 & \cdots & b_{n} \end{psmallmatrix}$ and $D$ is any diagonal matrix, then nor $N$  neither $-N$ is  positive semi-definite. Set  $(d_1,\cdots,d_n)$ as the diagonal entries of $D^*D, $ if $a_{i}+b_{i}\ge 0$ and  $a_ib_i-d_i < 0$ for all $i\le n,$              then  $\|N\|>\|A+B\|.$ for all symmetric norms
\end{lemma}
\begin{proof}The diagonal  of $N$ has  negative and   positive  numbers, thus nor $N$  neither $-N$ is  positive semi-definite, now any two diagonal matrices will commute, in particular $D^*$ and $A,$ by applying \tref{sinj} we get that  the eigenvalues of $N$ are the roots of $$\det((A-{\mu}I_{n})(B-{\mu}I_{n})-D^*D)=0$$
Equivalently the eigenvalues are all the  solutions  of the $n$ equations:$$ \begin{array}{rlrcl}1)&&({a}_1-{\mu})({b}_1-{\mu})-d_{1}&=&0
\\[0.1cm]2)& &({a}_2-{\mu})({b}_2-{\mu})-d_{2}&=&0
\\[0.1cm]3)&&\text{  }({a}_3-{\mu})({b}_3-{\mu})-d_{3}&=&0
\\\vdots{\text{ }}&&&\vdots&\\
i)&&({a}_i-{\mu})({b}_i-{\mu})-d_{n}&=&0\\\vdots\text{  }&&&\vdots&\\[0.1cm]
n)&&({a}_n-{\mu})({b}_n-{\mu})-d_{n}&=&0
\end{array}\indent (S)\\$$
 Let us denote by $x_i$ and $y_i$  the two solutions of the {$i^{th}$} equation then:
$$\begin{array}{rclcl}x_1+y_1&=&a_1+b_1&\ge&0\\[0.1cm]
x_2+y_2&=&a_2+b_2&\ge& 0\\ &\vdots&&\vdots&\\[0.1cm]
x_n+y_n&=&a_n+b_n&\ge &0\\[0.1cm]
\end{array}\indent \begin{array}{rclcl}x_1y_1&=&a_1b_1-d_1&<&0\\[0.1cm]
x_2y_2&=&a_2b_2-d_2&<& 0\\ &\vdots& &\vdots&\\[0.1cm]
x_ny_n&=&a_nb_n-d_n&< &0\\[0.1cm]
\end{array}. $$
This implies that each equation of $(S)$ has one negative  and one positive solution,  their sum is positive,  thus the positive root is bigger or equal than the negative one. Since  $A+B=\begin{psmallmatrix}a_{1}+b_{1} &\cdots & 0\\ \vdots &  \ddots &\vdots \\
0&  \cdots &a_{n}+ b_{n} \end{psmallmatrix},$  summing over  indexes we see that $\|N\|_k>\|A+B\|_k$ for $k=1,\cdots,n$ which yields to   $\|N\|>\|A+B\|$ for all symmetric norms.
\end{proof}
It seems easy to construct examples of  non P.S.D matrices $N$, such that $\|N\|_{s}>\|A+B\|_{s},$ let us have a  look of such inequality for P.S.D. matrices.
\begin{example}
Let $$N_y=\begin{bmatrix}  2& 0&0&2\\0&y&0&0\\0& 0& 1&0\\2&0& 0&{2}\end{bmatrix}=\begin{bmatrix}A&X\\X^*&B\end{bmatrix}$$ where  $A=\begin{bmatrix}2 &0\\ 0&y\end{bmatrix}$ and    $B=\begin{bmatrix}1&0\\0&2\end{bmatrix}.$   The eigenvalues of $N_y$ are the numbers: $\displaystyle{\lambda}_1=4,\text{   } {\lambda}_2=1,\text{       }{\lambda}_3=y, \text{   }{\lambda}_4=0,$ thus if $y\ge 0,$ $N_y$ is positive semi-definite  and for all $ y$ such that   $0\le y < 1$ we have
\begin{enumerate}
\item $4={\|N_y\|}_s>{\|A+B\|}_s=3.$
\item $ 16 +y^2+1={\| N\|}_{(2)}^2>{\|A+B\|}_{(2)}^2=4(3+y)+y^2+1.$
\end{enumerate}
\end{example}
\bibliographystyle{amsplain}

\end{document}